\documentclass[12pt,twoside]{article}
\usepackage{amsmath,amsthm,amssymb,amscd,ascmac, amsfonts}
\usepackage{mathrsfs}
\usepackage{stmaryrd}
\usepackage{braket}
\usepackage{accents}

\numberwithin{equation}{section}
\usepackage[dvips]{graphicx,color,psfrag}

\makeatletter
\newcommand{\figcaption}[1]{\def\@captype{figure}\caption{#1}}
\newcommand{\tblcaption}[1]{\def\@captype{table}\caption{#1}}
\makeatother

\makeatletter

\ifcase\@ptsize
\def\rpkern{\mathchoice{\kern-1.45em}{\kern-1.11em}{}{}}%
\def\grpkern{\mathchoice{\kern-1.013em}{\kern-0.825em}{}{}}%
\or
\def\rpkern{\mathchoice{\kern-1.44em}{\kern-1.11em}{}{}}%
\def\grpkern{\mathchoice{\kern-1.00em}{\kern-0.81em}{}{}}%
\or
\def\rpkern{\mathchoice{\kern-1.472em}{\kern-1.14em}{}{}}%
\def\grpkern{\mathchoice{\kern-1.00em}{\kern-0.815em}{}{}}%
\fi

\def\minibullet{\mathchoice%
{\raise0.2ex\hbox{$\scriptstyle\bullet$}}%
{\raise0.26ex\hbox{$\scriptscriptstyle\bullet$}}{}{}}
\def\butabullet{\mathchoice%
{\raise0.8ex\hbox{$\scriptstyle\bullet$}{\kern-0.365em}%
\lower0.4ex\hbox{$\scriptstyle\bullet$}}%
{\raise0.75ex\hbox{$\scriptscriptstyle\bullet$}{\kern-0.335em}%
\lower0.25ex\hbox{$\scriptscriptstyle\bullet$}}{}{}}

\def\customprod#1#2%
{\operatorname{\coprod\kern#2{#1}\kern#2\prod}\displaylimits}

\renewcommand{\Re}{\mathrm{Re}\,}
\renewcommand{\Im}{\mathrm{Im}\,}
\newcommand{\Res}[1]{\underset{#1}{\mathrm{Res}}\,}

\newcommand{\sgn}{\mathrm{sgn}\,}

\newtheorem*{multitheorem}{\variable@name}

\theoremstyle{definition}
\newcommand{\variable@name}{Theorem}
\newtheorem*{multiproclaim}{\variable@name}

\theoremstyle{plain}
\newtheorem{thm}{Theorem}[section]

\newtheorem{lem}[thm]{Lemma}

\newtheorem{prob}[thm]{Problem}

\theoremstyle{definition}

\newtheorem{rmk}[thm]{Remark}
\newtheorem{exa}[thm]{Example}

\begin{document}
\title{New trigonometric identities and reciprocity laws of generalized Dedekind sums}
\author{Genki Shibukawa}
\date{
\small MSC classes\,:\,11L03, 11F20}
\pagestyle{plain}

\maketitle


\begin{abstract}
We obtain new trigonometric identities, which are some product-to-sum type formulas for the higher derivatives of the cotangent and cosecant functions. 
Further, from specializations of our formulas, we derive not only various known reciprocity laws of generalized Dedekind sums but also new reciprocity laws of generalized Dedekind sums. 
\end{abstract}

\section{Introduction}
From Dedekind, so-called {\it{Dedekind sums}} and their reciprocity laws have been studied by the distinguished mathematicians. 
For an overview of previously defined generalized Dedekind sums, we refer to a good interpretation by M.\,Beck (see Section\,1 and 2 in \cite{B}). 
Let $\cot^{(m)}$ denote the $m$-th derivative of the cotangent function. 
In \cite{B}, for $a_{0},a_{1},\ldots,a_{r},m_{0},m_{1},\ldots,m_{r} \in \mathbb{Z}_{\geq 1}, w_{0},w_{1},\ldots,w_{r} \in \mathbb{C}$, Beck introduced Dedekind cotangent sums
$$
\frac{1}{a_{0}^{m_{0}}}\sum_{k \,{\rm{mod}}\,a_{0}}\prod_{j=1}^{r}\cot^{(m_{j}-1)}\left(\!\pi \! \left(a_{j}\frac{k+w_{0}}{a_{0}}-w_{j}\!\right)\!\right),
$$
where the sum is taken over $k \,{\rm{mod}}\,a_{0}$ for which the summand is not singular. 
The Dedekind cotangent sums include as special cases various generalizations of Dedekind sums expressed by the cotangent functions and their higher derivatives. 
Moreover, under some conditions for $a_{0},\ldots,a_{r},w_{0},\ldots,w_{r}$, Beck computed the residue of 
$$
\cot^{(m_{0}-1)}(\pi (a_{0}z-w_{0}))\prod_{l=1}^{r}\cot^{(m_{l}-1)}(\pi (a_{l}z-w_{l}))
$$
and derived various reciprocity laws of the Dedekind cotangent sums, which are not only known results by Dedekind, Rademacher, Apostol, Carlitz, Mikol\'{a}s, Dieter, Zagier, but also the truly new ones. 
However, since his method needs a case analysis based on some conditions for singular points of an integrand function of residue calculus, we have to prove the reciprocity laws individually.

On the other hand, an analogue of the Dedekind sum which was formed by replacing the cotangent functions in the Dedekind sum by the cosecant functions, and its reciprocity laws were introduced and proved by Fukuhara \cite{F2}. 
For example, Fukuhara treated the following type formulas. Let $p$ and $q$ are relatively prime positive integers. 

\noindent
{\rm{(0)}}\,{\rm{(Proposition\,1.3 in \cite{F} or (1.1) in \cite{F2})}}
For any complex number $z$, 
\begin{align}
\label{eq:prot formula}
pq\cot{(pz)}\cot{(qz)}&=-\cot^{(1)}(z)-pq+q\sum_{\mu =1}^{p-1}\cot{\left(\frac{\pi q \mu }{p}\right)}\cot{\left(z-\frac{\pi \mu}{p}\right)} \nonumber \\
& \quad 
+p\sum_{\mu =1}^{q-1}\cot{\left(\frac{\pi p \mu}{q}\right)}\cot{\left(z-\frac{\pi \mu}{q}\right)}.
\end{align}

\noindent
{\rm{(1)}}\,{\rm{((1.2) in \cite{F2})}} 
If $q$ is even, then
\begin{align}
pq\cot{(pz)}\csc{(qz)}&=-\cot^{(1)}{(z)}+q\sum_{\mu =1}^{p-1}\csc{\left(\frac{\pi q \mu}{p}\right)}\cot{\left(z-\frac{\pi \mu}{p}\right)} \nonumber \\
\label{eq:another prot thm1.3}
& \quad +p\sum_{\mu =1}^{q-1}(-1)^{\mu}\cot{\left(\frac{\pi p \mu }{q}\right)}\cot{\left(z-\frac{\pi \mu}{q}\right)}.
\end{align}

\noindent
{\rm{(2)}}\,{\rm{((1.4) in \cite{F2})}} 
If $q$ is odd, then
\begin{align}
pq\cot{(pz)}\csc{(qz)}&=-\csc^{(1)}{(z)}+q\sum_{\mu =1}^{p-1}\csc{\left(\frac{\pi q \mu}{p}\right)}\csc{\left(z-\frac{\pi \mu}{p}\right)} \nonumber \\
\label{eq:another prot thm1.4}
& \quad +p\sum_{\mu =1}^{q-1}(-1)^{\mu}\cot{\left(\frac{\pi p \mu }{q}\right)}\csc{\left(z-\frac{\pi \mu}{q}\right)}.
\end{align}

\noindent
{\rm{(3)}}\,{\rm{((1.3) in \cite{F2})}} 
If $p+q$ is even, then 
\begin{align}
pq\csc{(pz)}\csc{(qz)}&=-\cot^{(1)}{(z)}+q\sum_{\mu =1}^{p-1}(-1)^{\mu}\csc{\left(\frac{\pi q \mu }{p}\right)}\cot{\left(z-\frac{\pi \mu}{p}\right)} \nonumber \\
\label{eq:another prot thm1.1}
& \quad +p\sum_{\mu =1}^{q-1}(-1)^{\mu}\csc{\left(\frac{\pi p \mu}{q}\right)}\cot{\left(z-\frac{\pi \mu}{q}\right)}. 
\end{align}

\noindent
{\rm{(4)}}\,{\rm{((1.5) in \cite{F2})}} 
If $p+q$ is odd, then
\begin{align}
pq\csc{(pz)}\csc{(qz)}&=-\csc^{(1)}{(z)}+q\sum_{\mu =1}^{p-1}(-1)^{\mu}\csc{\left(\frac{\pi q \mu }{p}\right)}\csc{\left(z-\frac{\pi \mu}{p}\right)} \nonumber \\
\label{eq:another prot thm1.2}
& \quad +p\sum_{\mu =1}^{q-1}(-1)^{\mu}\csc{\left(\frac{\pi p \mu}{q}\right)}\csc{\left(z-\frac{\pi \mu}{q}\right)}, 
\end{align}
where $\csc^{(1)}{(z)}$ is the derivative of $\csc{(z)}$.

He also pointed out that these formulas can be regarded as a one parameter deformation of the reciprocity laws of some Dedekind sums, or a generating function of the reciprocity laws of some Dedekind-Apostol sums. 
Actually, in (\ref{eq:prot formula}), by comparing the coefficients of the Laurent expansion of (\ref{eq:prot formula}) at $z=0$, 
we obtain the reciprocity laws of the Dedekind-Apostol sums 
$$
s_{N}(q;p):=\frac{1}{2^{N+1}p}\sum_{\mu =1}^{p-1}\cot{\left(\frac{\pi q \mu}{p}\right)}\cot^{(N-1)}{\left(\frac{\pi \mu}{p}\right)} 
$$
as follows. 
\begin{align}
s_{1}(q;p)+s_{1}(p;q)&=\frac{p^{2}+q^{2}+1-3pq}{12pq}, \\
s_{2k+1}(q;p)+s_{2k+1}(p;q)
&=\frac{1}{2pq}\frac{B_{2k+2}}{k+1}+\frac{B_{2k+2}}{(2k+1)(2k+2)}(p^{2k+1}q^{-1}+p^{-1}q^{2k+1}) \nonumber \\
& \quad -(2k)!\sum_{l=1}^{k}\frac{B_{2l}B_{2k+2-2l}}{(2l)!(2k+2-2l)!}p^{2l-1}q^{2k+1-2l}. 
\end{align}
Here, $k$ is a positive integer and $\{B_{m}\}_{m=0,1,\ldots}$ are the Bernoulli numbers defined by 
$$
\frac{t}{e^{t}-1}=\sum_{m=0}^{\infty}\frac{B_{m}}{m!}t^{m}.
$$

As described above, from some product-to-sum type formulas for some trigonometric functions, like (\ref{eq:prot formula}), 
we easily obtain the reciprocity laws for various generalized Dedekind sums. 
In this article, taking into account the investigations, we present a detailed calculation of 
$$
\prod_{l=1}^{j_{I}}a_{l}^{m_{l}}\cot^{(m_{l}-1)}(\pi (a_{l}z-w_{l}))
\prod_{l=j_{I}+1}^{j_{I}+j_{I\hspace{-.1em}I}}a_{l}^{m_{l}}\csc^{(m_{l}-1)}(\pi (a_{l}z-w_{l}))
$$
and give a sum expression of the higher derivatives for the cotangent and cosecant functions, which can be regarded as a product-to-sum type formula for the higher derivatives of the cotangent and cosecant functions. 
We prove it under the completely generic condition, and only use Liouville's theorem and limit of some periodic functions at $z \to i\infty$. 
Thus, our proof is more generic than the method of Beck, and much simpler than Fukuhara's proof which needs some non-trivial trigonometric identities. 
Furthermore, from various specializations of our formula, we derive various reciprocity laws of the generalized Dedekind sums uniformly, which include the results in \cite{B} and \cite{F2} et al..

Let us now describe the content in this article. 
In Section\,2, we introduce the main object $\varphi_{N}^{(I)}(z)$ and $\varphi_{N}^{(I\hspace{-.1em}I)}(z)$ instead of the higher derivatives of the cotangent and cosecant functions, and recall their fundamental properties. 
In Section\,3 which is the main part of this article, under the general situation for the parameters, we provide a product-to-sum type formula for 
$$
\prod_{l=1}^{j_{I}}a_{l}^{m_{l}}\varphi_{m_{l}}^{(I)}(a_{l}z-w_{l})
\prod_{l=j_{I}+1}^{j_{I}+j_{I\hspace{-.1em}I}}a_{l}^{m_{l}}\varphi_{m_{l}}^{(I\hspace{-.1em}I)}(a_{l}z-w_{l})
$$
and derive new generalized reciprocity laws by writing down some specializations of the main theorem. 
In Section\,4, we restrict parameters of our main results in Section\,3 and give more explicit expression of our reciprocity laws. 
By these specializations, we show that our main results contain a lot of formulas for the generalized Dedekind sums by the distinguished mathematicians. 
Finally, in Section\,5, we present a future work for a variation on a theme of our formulas.

\section{Preliminaries}
Throughout the paper, we denote the ring of rational integers by $\mathbb{Z}$, 
the field of real numbers by $\mathbb{R}$, the field of complex numbers by $\mathbb{C}$ and $i:=\sqrt{-1}$. 
Further we use the notation:  
$$
\mathfrak{R}:=\{z \in \mathbb{C} \mid 0\leq \Re{z}< 1 \}.
$$
First, from Walker's book \cite{W}, we recall the two kinds of the periodic functions which play central roles in this article. 
For a positive integer $N$, we define the periodic functions by 
\begin{equation}
\varphi_{N}^{(J)}(z):=\frac{1}{z^{N}}+\sum_{n=1}^{\infty}(-1)^{n\delta_{J,I\hspace{-.1em}I}}\left(\frac{1}{(z+n)^{N}}+\frac{1}{(z-n)^{N}}\right)\,\,\,\,(J=I,\,I\hspace{-.1em}I).
\end{equation}
In \cite{W},  $\varphi_{N}^{(I)}(z)$ and $\varphi_{N}^{(I\hspace{-.1em}I)}(z)$ are denoted by $E_{N}(z)$ and $G_{N}(z)$ respectively. 
In the following, we list the main properties of $\varphi_{N}^{(J)}(z)$ ($J=I,I\hspace{-.1em}I$) from \cite{W}.

\noindent
\underline{Periodicity}\,
For any $\mu \in \mathbb{Z}$, 
\begin{equation}
\label{eq:period}
\varphi_{N}^{(J)}(z+\mu)=(-1)^{\mu\delta_{J,I\hspace{-.1em}I}}\varphi_{N}^{(J)}(z).
\end{equation}

\noindent
\underline{Derivation}\,
For any $N \in \mathbb{Z}_{\geq 0}$, 
\begin{equation}
\varphi_{N+1}^{(J)}(z)=\frac{(-1)^{N}}{N!}\left(\frac{d}{dz}\right)^{N}\varphi_{1}^{(J)}(z).
\end{equation}
In particular, 
\begin{equation}
\frac{d\varphi_{N}^{(J)}}{dz}(z)=-N\varphi_{N+1}^{(J)}(z).
\end{equation}

\noindent
\underline{Laurent expansions}\,
Let $\zeta(s)$ be the Riemann zeta function and 
\begin{align}
\alpha_{\mu}^{(I)}
&:=
\begin{cases}
    2\zeta(\mu)=(-1)^{\frac{\mu}{2}+1}\frac{B_{\mu}}{\mu !}(2\pi)^{\mu}  & (\text{if $\mu$ is even}) \\
    0  & (\text{if $\mu$ is odd})
\end{cases}, \\
\alpha_{\mu}^{(I\hspace{-.1em}I)}
&:=
\begin{cases}
    2(1-2^{1-\mu})\zeta(\mu)=2(2^{\mu-1}-1)(-1)^{\frac{\mu}{2}+1}\frac{B_{\mu}}{\mu !}\pi^{\mu} & (\text{if $\mu$ is even}) \\
    0  & (\text{if $\mu$ is odd})
\end{cases}.
\end{align}
Then, around $z=0$, we have
\begin{equation}
\label{eq:Laurent expansion prot}
\varphi_{N}^{(J)}(z)=\frac{1}{z^{N}}+(-1)^{N}\sum_{\nu \geq 0}\binom{N+\nu -1}{N-1}\alpha_{N+\nu}^{(J)}z^{\nu},  
\end{equation}
where $\binom{N+\nu -1}{N-1}$ is the binomial coefficient. 
More generally, the following result holds. 
For $X \subset \mathbb{C}$, we put 
$$
\delta_{X}(z):= \begin{cases}
    1 & ({\text{if $z \in X$}} ) \\
    0 & ({\text{if $z \not\in X$}} )
  \end{cases}.
$$
and define the signature by
\begin{equation}
\label{eq:def of sgn}
\sgn^{(J)}(z_{0} \,;a,w):=(-1)^{(az_{0}-w)\delta_{J,I\hspace{-.1em}I}}\delta_{\mathbb{Z}}(az_{0}-w).
\end{equation}
By the periodicity (\ref{eq:period}), if $\delta_{\mathbb{Z}}(az_{0}-w)=1$, then 
$$
\varphi_{N}^{(J)}(z-(az_{0}-w))=\sgn^{(J)}(z_{0} \,;a,w)\varphi_{N}^{(J)}(z).
$$
\begin{lem}
For any $a,m \in \mathbb{Z}_{\geq 1}$ and $w,z_{0} \in \mathbb{C}$, we have
\begin{equation}
\label{eq:fund prop}
a^{m}\varphi_{m}^{(J)}(az-w)
=\sgn^{(J)}(z_{0} \,;a,w)\left(z-z_{0}\right)^{-m}
+\sum_{\nu \geq 0}
A_{\nu}^{(J)}(z_{0} \,;a,m,w)\left(z-z_{0}\right)^{\nu}.
\end{equation}
Here, 
\begin{align}
(m)_{\nu}:=&\begin{cases}
    m(m+1)\cdots(m+\nu -1) & ({\text{if $\nu \geq 1$}}) \\
    1 & ({\text{if $\nu =0$}}) 
  \end{cases}. \nonumber \\
A_{\nu}^{(J)}(z_{0} \,;a,m,w)
:=&(-1)^{m}\sgn^{(J)}(z_{0} \,;a,w)\binom{m+\nu-1}{m-1}\alpha_{m+\nu}^{(J)}a^{m+\nu}\delta_{\mathbb{Z}}(az_{0}-w) \nonumber \\
&+(-1)^{\nu}\frac{(m)_{\nu}}{\nu!}a^{m+\nu}\Res{z=z_{0}}\frac{\varphi_{m+\nu}^{(J)}\left(az-w \right)}{z-z_{0}}\,dz\,(1-\delta_{\mathbb{Z}}(az_{0}-w)) \nonumber \\
=& \begin{cases}
    (-1)^{m}\sgn^{(J)}(z_{0} \,;a,w)\binom{m+\nu-1}{m-1}\alpha_{m+\nu}^{(J)}a^{m+\nu} & ({\text{if $\delta_{\mathbb{Z}}(az_{0}-w)=1$}}) \\
    (-1)^{\nu}\frac{(m)_{\nu}}{\nu!}\varphi_{m+\nu}^{(J)}\left(az_{0}-w \right)a^{m+\nu} & ({\text{if $\delta_{\mathbb{Z}}(az_{0}-w)=0$}}) 
  \end{cases}. \nonumber
\end{align}
\end{lem}
\begin{proof}
If $az_{0}-w \in \mathbb{C}\setminus \mathbb{Z}$, that means $\delta_{\mathbb{Z}}(az_{0}-w)=0$, then $z_{0}$ is not a pole of $a^{m}\varphi_{m}^{(J)}(az-w)$ and  
$$
\left(\frac{d}{dz}\right)^{\nu}a^{m}\varphi_{m}^{(J)}(az-w)\biggm|_{z=z_{0}}
=(-1)^{\nu}(m)_{\nu}\varphi_{m+\nu}^{(J)}(az_{0}-w)a^{m+\nu}.
$$
Thus, from the Taylor expansion of $a^{m}\varphi_{m}^{(J)}(az-w)$ at $z=z_{0}$, we have
$$
a^{m}\varphi_{m}^{(J)}(az-w)=\sum_{\nu \geq 0}(-1)^{\nu}\frac{(m)_{\nu}}{\nu!}\varphi_{m+\nu}^{(J)}(az_{0}-w)a^{m+\nu}(z-z_{0})^{\nu}.
$$
If $az_{0}-w \in \mathbb{Z}$, that is $\delta_{\mathbb{Z}}(az_{0}-w)=1$ case, then there exists $\mu \in \mathbb{Z}$ such that 
$$
z_{0}=\frac{w+\mu}{a}.
$$
Hence, by using the periodicity (\ref{eq:period}) and the Laurent expansion at $z_{0}=0$ (\ref{eq:Laurent expansion prot}), we have
\begin{align}
a^{m}\varphi_{m}^{(J)}(az-w)
&=\sgn^{(J)}(z_{0} \,;a,w)a^{m}\varphi_{m}^{(J)}(az-w-(az_{0}-w)) \nonumber \\
&=\sgn^{(J)}(z_{0} \,;a,w)a^{m}\varphi_{m}^{(J)}(a(z-z_{0})) \nonumber \\
&=\sgn^{(J)}(z_{0} \,;a,w)\left(z-z_{0}\right)^{-m} \nonumber \\
& \quad +(-1)^{m}\sgn^{(J)}(z_{0} \,;a,w)\sum_{\nu \geq 0}\binom{m+\nu-1}{m-1}\alpha_{m+\nu}^{(J)}a^{m+\nu}(z-z_{0})^{\nu}. \nonumber
\end{align}
\end{proof}

\noindent
\underline{Relationship with the cotangent and cosecant functions}\,
\begin{equation}
\varphi_{1}^{(I)}(z)=\pi \cot(\pi z),\,\,
\varphi_{1}^{(I\hspace{-.1em}I)}(z)=\pi \csc(\pi z).
\end{equation}


\noindent
\underline{Limit at $z \to i\infty$}\,
\begin{lem}
\begin{equation}
\label{eq:limit of Phi 0}
\lim_{z \to i\infty}\varphi_{N}^{(J)}(z)=-\pi{i}\delta_{N,1}\delta_{J,I}.
\end{equation}
\end{lem}
\begin{proof}
For $N\geq 2$, since $\varphi_{N}^{(J)}(z)$ is absolutely convergent, $\lim_{z \to i\infty}\varphi_{N}^{(J)}(z)=0$. 
If $N=1$, then 
\begin{align}
\lim_{z \to i\infty}\varphi_{1}^{(I)}(z)&=\lim_{z \to i\infty}\pi \cot{(\pi z)}=-\pi {i}, \nonumber \\
\lim_{z \to i\infty}\varphi_{1}^{(I\hspace{-.1em}I)}(z)&=\lim_{z \to i\infty}\pi \csc{(\pi z)}=0. \nonumber
\end{align}
\end{proof}

\section{Main results}
Let $r \in \mathbb{Z}_{\geq 2}, [r]:=\{1,\ldots,r\}, \mathbf{a}:=(a_{1},\ldots,a_{r}),\mathbf{m}:=(m_{1},\ldots,m_{r}) \in \mathbb{Z}_{\geq 1}^{r}, \mathbf{w}=(w_{1},\ldots,w_{r}) \in \mathfrak{R}^{r},\mathbf{j}=(j_{I},j_{I\hspace{-.1em}I}) \in \mathbb{Z}_{\geq 0}^{2}$. 
Here, suppose that $j_{I}+j_{I\hspace{-.1em}I}=r$. 
Further, we put   
\begin{align}
R_{\rho}&:=(R_{\rho}(\mathbf{a},\mathbf{w})=)\{\Lambda \subset [r] \mid \delta_{\mathbb{Z}}(a_{\lambda}\rho -w_{\lambda})=1 \,(\,{\text{for all}}\, \lambda \in \Lambda)\}, \nonumber \\
\Lambda^{c}&:=[r]\setminus \Lambda, \nonumber \\ 
K_{n,\Lambda}^{\pm}&:=\left\{(\nu_{k})_{k \in \Lambda^{c}}\in \mathbb{Z}_{\geq 0}^{|\Lambda^{c}|} \Bigg\vert n=\pm\left(\sum_{k \in \Lambda^{c}}\nu_{k} -\sum_{\lambda \in \Lambda}m_{\lambda}\right)\right\}, \nonumber \\
\delta^{(\mathbf{j},l)}_{I}&:=\sum_{j=1}^{j_{I}}\delta_{j,l}=\begin{cases}
    1 & ({\text{if }} 1\leq l\leq j_{I}) \\
    0 & ({\text{otherwise}})
  \end{cases}, \,\,
\delta^{(\mathbf{j},l)}_{I\hspace{-.1em}I}:=\sum_{j=j_{I}+1}^{j_{I}+j_{I\hspace{-.1em}I}}\delta_{j,l}=\begin{cases}
    1 & ({\text{if }} j_{I}+1\leq l\leq j_{I}+j_{I\hspace{-.1em}I}) \\
    0 & ({\text{otherwise}})
  \end{cases}, \nonumber
\end{align}
and
\begin{align}
\sgn^{(\mathbf{j},l)}(z_{0} \,;a,w)
:=&\sgn^{(I)}(z_{0} \,;a,w)\delta^{(\mathbf{j},l)}_{I}+\sgn^{(I\hspace{-.1em}I)}(z_{0} \,;a,w)\delta^{(\mathbf{j},l)}_{I\hspace{-.1em}I} \nonumber \\
=&(-1)^{(az_{0}-w)\delta^{(\mathbf{j},l)}_{I\hspace{-.1em}I}}\delta_{\mathbb{Z}}(az_{0}-w), \nonumber \\
\varphi_{N}^{(\mathbf{j},l)}(z)
:=&\varphi_{N}^{(I)}(z)\delta^{(\mathbf{j},l)}_{I}+\varphi_{N}^{(I\hspace{-.1em}I)}(z)\delta^{(\mathbf{j},l)}_{I\hspace{-.1em}I}, \nonumber \\
\alpha_{\nu}^{(\mathbf{j},l)}
:=& \alpha_{\nu}^{(I)}\delta^{(\mathbf{j},l)}_{I}+\alpha_{\nu}^{(I\hspace{-.1em}I)}\delta^{(\mathbf{j},l)}_{I\hspace{-.1em}I}, \nonumber \\
A_{\nu}^{(\mathbf{j},l)}(z_{0} \,;a,m,w)
:=&A_{\nu}^{(I)}(z_{0} \,;a,m,w)\delta^{(\mathbf{j},l)}_{I}+A_{\nu}^{(I\hspace{-.1em}I)}(z_{0} \,;a,m,w)\delta^{(\mathbf{j},l)}_{I\hspace{-.1em}I}. \nonumber
\end{align}
Moreover, for convenience, we consider the following two cases according to $\mathbf{a}$ and $\mathbf{j}$.  
\begin{align}
{\text{Case.\,$I$}}&:j_{I\hspace{-.1em}I}=0,\,\,{\text{or}}\,\,\sum_{l=j_{I}+1}^{r}a_{l} \,\,\,\text{is even.} \nonumber \\
{\text{Case.\,$I\hspace{-.1em}I$}}&:\sum_{l=j_{I}+1}^{r}a_{l} \,\,\,\text{is odd.} \nonumber 
\end{align}
The following theorem is the main result of this article. 
\begin{thm}
\label{prop:main theorem}
Let 
$$
\Phi^{(\mathbf{j})}(z\,;\mathbf{a},\mathbf{m},\mathbf{w})
:=\prod_{l=1}^{r}a_{l}^{m_{l}}\varphi_{m_{l}}^{(\mathbf{j},l)}(a_{l}z-w_{l}).
$$
Here, if a $j_{J}$ is zero, we omit the product term for $\varphi_{m_{l}}^{(J)}(a_{l}z-w_{l})$. 
For Case.\,$J$, we have
\begin{align}
\Phi^{(\mathbf{j})}(z \,;\mathbf{a},\mathbf{m},\mathbf{w})&=\cos\left(\frac{\pi{r}}{2}\right)\pi^{r}\delta_{j_{I\hspace{-.1em}I},0}\prod_{l=1}^{r}a_{l}\delta_{m_{l},1} \nonumber \\ 
\label{eq:main result}
& \quad +\sum_{n=1}^{|\mathbf{m}|}\sum_{\rho}\sum_{\Lambda \in R_{\rho}}\sum_{(\nu_{k})_{k \in \Lambda^{c}}\in K_{n,\Lambda}^{-}}
\prod_{l \in \Lambda}(-1)^{(a_{l}\rho -w_{l})\delta^{(\mathbf{j},l)}_{I\hspace{-.1em}I}} \nonumber \\
& \quad \cdot \prod_{u \in \Lambda^{c}}\{A_{\nu_{u}}^{(\mathbf{j},u)}(\rho \,;a_{u},m_{u},w_{u})\}
\varphi_{n}^{(J)}(z-\rho),
\end{align}
where $\rho$ runs over all poles of $\Phi^{(\mathbf{j})}(z \,;\mathbf{a},\mathbf{m},\mathbf{w})$ in $\mathfrak{R}$, and 
$$
|\mathbf{m}|:=m_{1}+\cdots+m_{r}.
$$
\end{thm}
\begin{proof}
We denote $\Psi^{(\mathbf{j})}(z \,;\mathbf{a},\mathbf{m},\mathbf{w})$ by the right hand side of (\ref{eq:main result}). 
We claim that for all Case.\,$I$,$I\hspace{-.1em}I$, 
$$
\Phi^{(\mathbf{j})}(z \,;\mathbf{a},\mathbf{m},\mathbf{w})-\Psi^{(\mathbf{j})}(z \,;\mathbf{a},\mathbf{m},\mathbf{w})=0.
$$
First, we consider the Laurent expansion of $\Phi^{(\mathbf{j})}(z \,;\mathbf{a},\mathbf{m},\mathbf{w})$ at $\rho$   
\begin{align}
\Phi^{(\mathbf{j})}(z \,;\mathbf{a},\mathbf{m},\mathbf{w})
&=\prod_{l=1}^{r}\left\{\sgn^{(\mathbf{j},l)}(\rho \,;a_{l},w_{l})(z-\rho)^{-m_{l}}
+\sum_{\nu_{l}\geq 0}
A_{\nu}^{(\mathbf{j},l)}(\rho \,;a_{l},m_{l},w_{l})(z-\rho)^{\nu_{l}}\right\} \nonumber \\
&=\prod_{l=1}^{r}\{\sgn^{(\mathbf{j},l)}(\rho \,;a_{l},w_{l})\}(z-\rho)^{-|\mathbf{m}|}
+\sum_{N=1}^{r-1}
\sum_{1\leq \lambda_{1} <\ldots<\lambda_{N} \leq r}  \nonumber \\
& \quad \cdot
\left\{\sum_{\substack{|\mathbf{m}|>\sum_{k=1}^{N}(m_{\lambda_{k}}+\nu_{\lambda_{k}}), \\ \nu_{\lambda_{1}},\ldots,\nu_{\lambda_{N}}\geq 0}}
+\sum_{\substack{|\mathbf{m}|\leq \sum_{k=1}^{N}(m_{\lambda_{k}}+\nu_{\lambda_{k}}), \\ \nu_{\lambda_{1}},\ldots,\nu_{\lambda_{N}}\geq 0}}\right\}
\prod_{l\in [r]\setminus \{\lambda_{1},\ldots,\lambda_{N} \}}\!\!\!\!\!\!\!\!\!\!\!
\sgn^{(\mathbf{j},l)}(\rho \,;a_{l},w_{l}) \nonumber \\
& \quad \cdot 
\prod_{u=1}^{N}A_{{\nu}_{\lambda_{u}}}^{(\mathbf{j},\lambda_{u})}(\rho \,;a_{\lambda_{u}},m_{\lambda_{u}},w_{\lambda_{u}})
(z-\alpha_{j})^{-|\mathbf{m}|+\sum_{k=1}^{N}(m_{\lambda_{k}}+\nu_{\lambda_{k}})} \nonumber \\
&=\sum_{n=1}^{|\mathbf{m}|}\sum_{\Lambda \in R_{\rho}}
\sum_{(\nu_{k})_{k \in \Lambda^{c}}\in K_{n,\Lambda}^{-}}\!\!
\prod_{l \in \Lambda}(-1)^{(a_{l}\rho -w_{l})\delta^{(\mathbf{j},l)}_{I\hspace{-.1em}I}}\!\! 
\prod_{u \in \Lambda^{c}}\!\!\{A_{\nu_{u}}^{(\mathbf{j},u)}(\rho \,;a_{u},m_{u},w_{u})\}(z-\rho)^{-n}
\nonumber \\
\label{Laurent exp of Phi}
& \quad +\sum_{\mu \geq 0}\sum_{\Lambda \in R_{\rho}}
\sum_{(\nu_{k})_{k \in \Lambda^{c}}\in K_{\mu,\Lambda}^{+}}
\prod_{l \in \Lambda}(-1)^{(a_{l}\rho -w_{l})\delta^{(\mathbf{j},l)}_{I\hspace{-.1em}I}} \nonumber \\
& \quad \cdot \prod_{u \in \Lambda^{c}}\{A_{\nu_{u}}^{(\mathbf{j},u)}(\rho \,;a_{u},m_{u},w_{u})\}(z-\rho)^{\mu}.
\end{align}
Further, from (\ref{eq:period}), we have
\begin{equation}
\label{eq:period of Phi}
\Phi^{(\mathbf{j})}(z+\mu \,;\mathbf{a},\mathbf{m},\mathbf{w})
=(-1)^{\mu \delta_{J,I\hspace{-.1em}I}}
\Phi^{(\mathbf{j})}(z \,;\mathbf{a},\mathbf{m},\mathbf{w}). 
\end{equation}
Thus, for Case.\,$I,I\hspace{-.1em}I$, $\Phi^{(\mathbf{j})}(z \,;\mathbf{a},\mathbf{m},\mathbf{w})$ and $\Phi_{N}^{(J)}(z)$ are the periodic functions with same period. 
In addition, by (\ref{eq:fund prop}) and (\ref{Laurent exp of Phi}), $\Phi^{(\mathbf{j})}(z \,;\mathbf{a},\mathbf{m},\mathbf{w})-\Psi^{(\mathbf{j})}(z \,;\mathbf{a},\mathbf{m},\mathbf{w})$ is entire.

Next, we remark that $\varphi_{N}^{(J)}(z)$ is bounded on the set $\mathfrak{R}_{1}:=\mathfrak{R}\cap \{z \in \mathbb{C} \mid |\Im{z}|\geq 1+2\max_{\rho}|\Im{\rho}|\}$  from (\ref{eq:limit of Phi 0}). 
Thus, $\Phi^{(\mathbf{j})}(z \,;\mathbf{a},\mathbf{m},\mathbf{w})-\Psi^{(\mathbf{j})}(z \,;\mathbf{a},\mathbf{m},\mathbf{w})$ is also bounded on $\mathfrak{R}_{1}$. 
Hence, by the periodicity, $\Phi^{(\mathbf{j})}(z \,;\mathbf{a},\mathbf{m},\mathbf{w})-\Psi^{(\mathbf{j})}(z \,;\mathbf{a},\mathbf{m},\mathbf{w})$ is bounded on $\mathbb{C}$. 
Then by the well-known Liouville's theorem, there exists a constant $c^{(\mathbf{j},J)}(\mathbf{a},\mathbf{m},\mathbf{w})$ such that 
$$
\Phi^{(\mathbf{j})}(z \,;\mathbf{a},\mathbf{m},\mathbf{w})-\Psi^{(\mathbf{j})}(z \,;\mathbf{a},\mathbf{m},\mathbf{w})
=c^{(\mathbf{j},J)}(\mathbf{a},\mathbf{m},\mathbf{w}). 
$$
If we restrict $z \in \mathbb{C}$ and $w_{1},\ldots,w_{r} \in \mathfrak{R}$ to $\mathbb{R}$, then $A_{\nu_{u}}^{(\mathbf{j},u)}(\rho \,;a_{u},m_{u},w_{u}), c^{(\mathbf{j},J)}(\mathbf{a},\mathbf{m},\mathbf{w}) \in \mathbb{R}$. 
In addition, we calculate
\begin{align}
\lim_{z \to i\infty}\Phi^{(\mathbf{j})}(z \,;\mathbf{a},\mathbf{m},\mathbf{w})
&=\prod_{l=1}^{r}a_{l}^{m_{l}}(-\pi{i})\delta_{m_{l},1}\delta^{(\mathbf{j},l)}_{I} \nonumber \\
&=(-i)^{r}{\pi}^{r}\delta_{j_{I\hspace{-.1em}I},0}\prod_{l=1}^{r}a_{l}\delta_{m_{l},1} \nonumber \\
\label{eq:limit of Phi}
&=\left\{\cos\left(\frac{\pi{r}}{2}\right)-i\sin\left(\frac{\pi{r}}{2}\right)\right\}{\pi}^{r}\delta_{j_{I\hspace{-.1em}I},0}\prod_{l=1}^{r}a_{l}\delta_{m_{l},1}, \\
\lim_{z \to i\infty}\Psi^{(\mathbf{j})}(z \,;\mathbf{a},\mathbf{m},\mathbf{w})&=\cos\left(\frac{\pi{r}}{2}\right)\pi^{r}\delta_{j_{I\hspace{-.1em}I},0}\prod_{l=1}^{r}a_{l}\delta_{m_{l},1} \nonumber \\
& \quad -{\pi}i\delta_{J,I}\sum_{\rho}\sum_{\Lambda \in R_{\rho}}\sum_{(\nu_{k})_{k \in \Lambda^{c}}\in K_{1,\Lambda}^{-}}
\prod_{l \in \Lambda}(-1)^{(a_{l}\rho -w_{l})\delta^{(\mathbf{j},l)}_{I\hspace{-.1em}I}} \nonumber \\
\label{eq:limit of Psi}
& \quad \cdot \prod_{u \in \Lambda^{c}}A_{\nu_{u}}^{(\mathbf{j},u)}(\rho \,;a_{u},m_{u},w_{u}).
\end{align}
Thus, we have
\begin{align}
c^{(\mathbf{j},J)}(\mathbf{a},\mathbf{m},\mathbf{w})&=
\Re{\{c^{(\mathbf{j},J)}(\mathbf{a},\mathbf{m},\mathbf{w})\}} \nonumber \\
&=\Re{\left\{\lim_{z \to i\infty}\{\Phi^{(\mathbf{j})}(z \,;\mathbf{a},\mathbf{m},\mathbf{w})-\Psi^{(\mathbf{j})}(z \,;\mathbf{a},\mathbf{m},\mathbf{w})\}\right\}}=0. \nonumber
\end{align}
\end{proof}
As a corollary of this theorem, we obtain the following theorem immediately.  
\begin{thm}
\label{prop:main theorem 2}
{\rm{(1)}}\,
We have
\begin{align}
\label{eq:main result 2}
& \sum_{\rho}\sum_{\Lambda \in R_{\rho}}\sum_{(\nu_{k})_{k \in \Lambda^{c}}\in K_{1,\Lambda}^{-}}
\prod_{l \in \Lambda}(-1)^{(a_{l}\rho -w_{l})\delta^{(\mathbf{j},l)}_{I\hspace{-.1em}I}}
\prod_{u \in \Lambda^{c}}A_{\nu_{u}}^{(\mathbf{j},u)}(\rho \,;a_{u},m_{u},w_{u}) \nonumber \\
& \quad =\pi^{r-1}\sin\left(\frac{\pi{r}}{2}\right)\delta_{j_{I\hspace{-.1em}I},0}\prod_{l=1}^{r}a_{l}\delta_{m_{l},1}.
\end{align}
{\rm{(2)}}\,For any $z_{0} \in \mathbb{C},\mu \in \mathbb{Z}_{\geq 0}$, we have
\begin{align}
& \sum_{n=1}^{|\mathbf{m}|}\sum_{\rho}\sum_{\Lambda \in R_{\rho}}\sum_{(\nu_{k})_{k \in \Lambda^{c}}\in K_{n,\Lambda}^{-}}
\!\prod_{l \in \Lambda}(-1)^{(a_{l}\rho -w_{l})\delta^{(\mathbf{j},l)}_{I\hspace{-.1em}I}}
\prod_{u \in \Lambda^{c}}\{A_{\nu_{u}}^{(\mathbf{j},u)}(\rho \,;a_{u},m_{u},w_{u})\}
A_{\mu}^{(J)}(z_{0} \,;1,n,\rho) \nonumber \\
\label{eq:main result 3}
&=-\cos\left(\frac{\pi{r}}{2}\right)\pi^{r}\delta_{\mu ,0}\delta_{j_{I\hspace{-.1em}I},0}\prod_{l=1}^{r}a_{l}\delta_{m_{l},1} \nonumber \\
& \quad +\sum_{\Lambda \in R_{z_{0}}}\sum_{(\nu_{k})_{k \in \Lambda^{c}}\in K_{\mu,\Lambda}^{+}}
\!\prod_{l \in \Lambda}\!(-1)^{(a_{l}z_{0} -w_{l})\delta^{(\mathbf{j},l)}_{I\hspace{-.1em}I}}
\!\prod_{u \in \Lambda^{c}}\!A_{\nu_{u}}^{(\mathbf{j},u)}(z_{0} \,;a_{u},m_{u},w_{u}). 
\end{align}
\end{thm}
\begin{proof}
{\rm{(1)}}\,It follows from (\ref{eq:limit of Phi}) and (\ref{eq:limit of Psi}) immediately.

\noindent
{\rm{(2)}}\,
We expand both sides of (\ref{eq:main result}) into the Laurent series of $z-z_{0}$ 
and compare the coefficients of $(z-z_{0})^{\mu}$ of both sides. 
Indeed, by replacing $\rho$ with $z_{0}$ in (\ref{Laurent exp of Phi}), we have
\begin{align}
\Res{z=z_{0}}\frac{\Phi^{(\mathbf{j})}(z \,;\mathbf{a},\mathbf{m},\mathbf{w})}{(z-z_{0})^{\mu +1}}\,dz
&=\sum_{\Lambda \in R_{z_{0}}}\sum_{(\nu_{k})_{k \in \Lambda^{c}}\in K_{\mu,\Lambda}^{+}}
\prod_{l \in \Lambda}(-1)^{(a_{l}z_{0} -w_{l})\delta^{(\mathbf{j},l)}_{I\hspace{-.1em}I}}
\prod_{u \in \Lambda^{c}}A_{\nu_{u}}^{(\mathbf{j},u)}(z_{0} \,;a_{u},m_{u},w_{u}). 
\nonumber
\end{align}
On the other hand, from (\ref{eq:main result}), 
\begin{align}
\Res{z=z_{0}}\frac{\Psi^{(\mathbf{j})}(z \,;\mathbf{a},\mathbf{m},\mathbf{w})}{(z-z_{0})^{\mu +1}}\,dz
&=\cos\left(\frac{\pi{r}}{2}\right)\pi^{r}\delta_{\mu ,0}\delta_{j_{I\hspace{-.1em}I},0}\prod_{l=1}^{r}a_{l}\delta_{m_{l},1} \nonumber \\
& \quad +\sum_{n=1}^{|\mathbf{m}|}\sum_{\rho}\sum_{\Lambda \in R_{\rho}}\sum_{(\nu_{k})_{k \in \Lambda^{c}}\in K_{n,\Lambda}^{-}}
\prod_{l \in \Lambda}(-1)^{(a_{l}\rho -w_{l})\delta^{(\mathbf{j},l)}_{I\hspace{-.1em}I}} \nonumber \\
& \quad \cdot \prod_{u \in \Lambda^{c}}\{A_{\nu_{u}}^{(\mathbf{j},u)}(\rho \,;a_{u},m_{u},w_{u})\}
A_{\mu}^{(J)}(z_{0} \,;1,n,\rho). \nonumber
\end{align}
Therefore, we have the conclusion. 
\end{proof} 
As we will see later, Theorem\,\ref{prop:main theorem 2} {\rm{(1)}} is a generalization of various reciprocity laws in \cite{B}. 
Theorem\,\ref{prop:main theorem 2} {\rm{(2)}} means (\ref{eq:main result}) is regarded as a generating function of the reciprocity laws (\ref{eq:main result 3}). 
Hence, this result and proof are generalizations of theorem\,1.2 in \cite{F} and its proof.

\begin{rmk}
For Theorems\,\ref{prop:main theorem}, \ref{prop:main theorem 2}, we also have the other expressions which are useful for writing down various specific examples. 
Let 
$$
d_{v}^{(\mu)}\!:=\!\#\!\left\{\!a_{j}\in \{a_{1},\ldots,a_{r}\}, w_{j}\in \{w_{1},\ldots,w_{r}\}, \mu_{j}\in \{0,\ldots, a_{j}-1\}  \Bigg\vert
\frac{w_{v}+\mu}{a_{v}}=\frac{w_{j}+\mu_{j}}{a_{j}}\right\}.
$$
Our main result (\ref{eq:main result}) becomes 
\begin{align}
\Phi^{(\mathbf{j})}(z \,;\mathbf{a},\mathbf{m},\mathbf{w})&=\cos\left(\frac{\pi{r}}{2}\right)\pi^{r}\delta_{j_{I\hspace{-.1em}I},0}\prod_{l=1}^{r}a_{l}\delta_{m_{l},1} \nonumber \\ 
& \quad +\sum_{n=1}^{|\mathbf{m}|}\sum_{v=1}^{r}\sum_{\mu_{v}=0}^{a_{v}-1}\frac{1}{d_{v}^{(\mu_{v})}}\sum_{\Lambda \in R_{\frac{w_{v}+\mu_{v}}{a_{v}}}}\sum_{(\nu_{k})_{k \in \Lambda^{c}}\in K_{n,\Lambda}^{-}}
\prod_{l \in \Lambda}(-1)^{\left(a_{l}\frac{w_{v}+\mu_{v}}{a_{v}} -w_{l}\right)\delta^{(\mathbf{j},l)}_{I\hspace{-.1em}I}} \nonumber \\
& \quad \cdot \prod_{u \in \Lambda^{c}}\left\{A_{\nu_{u}}^{(\mathbf{j},u)}\left(\frac{w_{v}+\mu_{v}}{a_{v}} \,;a_{u},m_{u},w_{u}\right)\right\}
\varphi_{n}^{(J)}\left(z-\frac{w_{v}+\mu_{v}}{a_{v}}\right). 
\end{align}
Similarly, (\ref{eq:main result 2}) and (\ref{eq:main result 3}) are  
\begin{align}
& \sum_{v=1}^{r}\sum_{\mu_{v}=0}^{a_{v}-1}\frac{1}{d_{v}^{(\mu_{v})}}\sum_{\Lambda \in R_{\frac{w_{v}+\mu_{v}}{a_{v}}}}\sum_{(\nu_{k})_{k \in \Lambda^{c}}\in K_{1,\Lambda}^{-}}
\prod_{l \in \Lambda}(-1)^{\left(a_{l}\frac{w_{v}+\mu_{v}}{a_{v}} -w_{l}\right)\delta^{(\mathbf{j},l)}_{I\hspace{-.1em}I}}
\prod_{u \in \Lambda^{c}}A_{\nu_{u}}^{(\mathbf{j},u)}\left(\frac{w_{v}+\mu_{v}}{a_{v}} \,;a_{u},m_{u},w_{u}\right) \nonumber \\
& \quad =\pi^{r-1}\sin\left(\frac{\pi{r}}{2}\right)\delta_{j_{I\hspace{-.1em}I},0}\prod_{l=1}^{r}a_{l}\delta_{m_{l},1}.
\end{align}
and for any $z_{0} \in \mathbb{C},\mu \in \mathbb{Z}_{\geq 0}$, 
\begin{align}
& \sum_{n=1}^{|\mathbf{m}|}\sum_{v=1}^{r}\sum_{\mu_{v}=0}^{a_{v}-1}\frac{1}{d_{v}^{(\mu_{v})}}\sum_{\Lambda \in R_{\frac{w_{v}+\mu_{v}}{a_{v}}}}\sum_{(\nu_{k})_{k \in \Lambda^{c}}\in K_{n,\Lambda}^{-}}
\prod_{l \in \Lambda}(-1)^{\left(a_{l}\frac{w_{v}+\mu_{v}}{a_{v}} -w_{l}\right)\delta^{(\mathbf{j},l)}_{I\hspace{-.1em}I}} \nonumber \\
& \prod_{u \in \Lambda^{c}}\left\{A_{\nu_{u}}^{(\mathbf{j},u)}\left(\frac{w_{v}+\mu_{v}}{a_{v}} \,;a_{u},m_{u},w_{u}\right)\right\}
A_{\mu}^{(J)}\left(z_{0} \,;1,n,\frac{w_{v}+\mu_{v}}{a_{v}}\right) \nonumber \\
&=-\cos\left(\frac{\pi{r}}{2}\right)\pi^{r}\delta_{\mu ,0}\delta_{j_{I\hspace{-.1em}I},0}\prod_{l=1}^{r}a_{l}\delta_{m_{l},1} \nonumber \\
& \quad +\sum_{\Lambda \in R_{z_{0}}}\sum_{(\nu_{k})_{k \in \Lambda^{c}}\in K_{\mu,\Lambda}^{+}}
\!\prod_{l \in \Lambda}\!(-1)^{(a_{l}z_{0} -w_{l})\delta^{(\mathbf{j},l)}_{I\hspace{-.1em}I}}
\!\prod_{u \in \Lambda^{c}}\!A_{\nu_{u}}^{(\mathbf{j},u)}(z_{0} \,;a_{u},m_{u},w_{u}). 
\end{align}
respectively. 

\end{rmk}

\section{Some special cases of the main theorem}
By specializing our main results, we derive various reciprocity laws of the generalized Dedekind sums. 

\subsection{Multiplicity free case}
In this subsection, we assume for all distinct $k,l \in [r]$ and $\mu_{k}=0,1,\ldots,a_{k}-1$, $\mu_{l}=0,1,\ldots,a_{l}-1$, 
$$
\frac{w_{k}+\mu_{k}}{a_{k}}\not=\frac{w_{l}+\mu_{l}}{a_{l}}.
$$
Under this condition, all poles of $a_{j}^{m_{j}}\varphi_{m_{j}}^{(J)}(a_{j}z-w_{j})$ for each $j=1,\ldots,r$ on $\mathfrak{R}$
$$
\frac{w_{j}+\mu_{j}}{a_{j}}\quad (j=1,\ldots,r,\,\, {\text{and}}\,\, \mu_{j}=0,\ldots,a_{j}-1)
$$
are multiplicity free, that means 
$$
\delta_{\mathbb{Z}}\left(a_{l}\frac{w_{j}+\mu_{j}}{a_{j}}-w_{l}\right)=
\begin{cases}
    1 & (\text{if $l=j$}) \\
    0 & (\text{if $l\not =j$})
\end{cases}, 
$$
and $d_{v}^{(\mu_{v})}=1$ for all $1\leq v\leq r$, $0\leq \mu_{v}\leq a_{v}-1$. 
Hence, Theorems\,\ref{prop:main theorem}, \ref{prop:main theorem 2} are as follows. 
\begin{thm}
\label{prop:main theorem multiplicity free}
We have
\begin{align}
\Phi^{(\mathbf{j})}(z \,;\mathbf{a},\mathbf{m},\mathbf{w})&=\cos\left(\frac{\pi{r}}{2}\right)\pi^{r}\delta_{j_{I\hspace{-.1em}I},0}\prod_{l=1}^{r}a_{l}\delta_{m_{l},1} \nonumber \\ 
& \quad +\sum_{n=1}^{\max_{j \in [r]}{\{m_{j}\}}}\sum_{l=1}^{r}\sum_{\mu_{l}=0}^{a_{l}-1}
\sum_{\substack{n=m_{l}-\sum_{1\leq k\not=l\leq r}\nu_{k}, \\ \nu_{1},\ldots,\nu_{r}\geq 0}}(-1)^{\mu_{l}\delta^{(\mathbf{j},l)}_{I\hspace{-.1em}I}} \nonumber \\
& \quad \cdot \prod_{1\leq u\not=l\leq r}\left\{(-1)^{\nu_{u}}\frac{(m_{u})_{\nu_{u}}}{\nu_{u}!}\varphi_{m_{u}+\nu_{u}}^{(\mathbf{j},u)}\left(a_{u}\frac{w_{l}+\mu_{l}}{a_{l}}-w_{u} \right)a_{u}^{m_{u}+\nu_{u}}\right\} \nonumber \\
\label{eq:mutiplicity free}
& \quad \cdot
\varphi_{n}^{(J)}\left(z-\frac{w_{l}+\mu_{l}}{a_{l}} \right).
\end{align}
\end{thm}
\begin{thm}
\label{prop:main theorem multiplicity free 2}
{\rm{(1)}}\,
We have
\begin{align}
& \sum_{l=1}^{r}\sum_{\mu_{l}=0}^{a_{l}-1}
\sum_{\substack{1=m_{l}-\sum_{1\leq k\not=l\leq r}\nu_{k}, \\ \nu_{1},\ldots,\nu_{r}\geq 0}}(-1)^{\mu_{l}\delta^{(\mathbf{j},l)}_{I\hspace{-.1em}I}} \nonumber \\
\label{eq:mutiplicity free 1}
& \quad \cdot \!\!\!\!\prod_{1\leq u\not=l\leq r}\!\!\left\{(-1)^{\nu_{u}}\frac{(m_{u})_{\nu_{u}}}{\nu_{u}!}\varphi_{m_{u}+\nu_{u}}^{(\mathbf{j},u)}\left(a_{u}\frac{w_{l}+\mu_{l}}{a_{l}}-w_{u} \right)a_{u}^{m_{u}+\nu_{u}}\right\}
\!=\!\pi^{r-1}\sin\left(\frac{\pi{r}}{2}\right)\delta_{j_{I\hspace{-.1em}I},0}\prod_{l=1}^{r}a_{l}\delta_{m_{l},1}.  
\end{align}
In particular, for $m_{1}=\cdots=m_{r}=1$, 
\begin{align}
\label{eq:mutiplicity free 1.2}
& \sum_{l=1}^{r}\sum_{\mu_{l}=0}^{a_{l}-1}(-1)^{\mu_{l}\delta^{(\mathbf{j},l)}_{I\hspace{-.1em}I}}
\prod_{1\leq u\not=l\leq r}\left\{\varphi_{1}^{(\mathbf{j},u)}\left(a_{u}\frac{w_{l}+\mu_{l}}{a_{l}}-w_{u} \right)a_{u}\right\}
=\pi^{r-1}\sin\left(\frac{\pi{r}}{2}\right)\delta_{j_{I\hspace{-.1em}I},0}\prod_{l=1}^{r}a_{l}.
\end{align}
{\rm{(2)}}\,
For any $\mu \in \mathbb{Z}_{\geq 0}$ and $z_{0} \in \mathfrak{R}$, 
\begin{align}
& \sum_{n=1}^{\max_{j \in [r]}{\{m_{j}\}}}\sum_{l=1}^{r}\sum_{\mu_{l}=0}^{a_{l}-1}
\sum_{\substack{n=m_{l}-\sum_{1\leq k\not=l\leq r}\nu_{k}, \\ \nu_{1},\ldots,\nu_{r}\geq 0}}(-1)^{\mu_{l}\delta^{(\mathbf{j},l)}_{I\hspace{-.1em}I}} \nonumber \\
& \cdot \prod_{1\leq u\not=l\leq r}\!\!\left\{(-1)^{\nu_{u}}\frac{(m_{u})_{\nu_{u}}}{\nu_{u}!}\varphi_{m_{u}+\nu_{u}}^{(\mathbf{j},u)}\!\left(\!a_{u}\frac{w_{l}+\mu_{l}}{a_{l}}-w_{u} \!\right)\!a_{u}^{m_{u}+\nu_{u}}\!\right\}A_{\mu}^{(J)}\left(z_{0}\,;1,n, \frac{w_{l}+\mu_{l}}{a_{l}} \right) \nonumber \\
\label{eq:mutiplicity free 2}
&=-\cos\left(\frac{\pi{r}}{2}\right)\pi^{r}\delta_{\mu ,0}\delta_{j_{I\hspace{-.1em}I},0}\prod_{l=1}^{r}a_{l}\delta_{m_{l},1} \nonumber \\
& \quad +\!\!\sum_{\Lambda \in R_{z_{0}}}\sum_{(\nu_{k})_{k \in \Lambda^{c}}\in K_{\mu,\Lambda}^{+}}
\!\prod_{l \in \Lambda}\!(-1)^{(a_{l}z_{0} -w_{l})\delta^{(\mathbf{j},l)}_{I\hspace{-.1em}I}}
\!\prod_{u \in \Lambda^{c}}\!A_{\nu_{u}}^{(\mathbf{j},u)}(z_{0} \,;a_{u},m_{u},w_{u}).
\end{align}
\end{thm}
\begin{exa}
When we consider the case of $(j_{I},j_{I\hspace{-.1em}I})=(r,0)$, (\ref{eq:mutiplicity free 1}) is none other than Beck's reciprocity (Theorem\,2 in \cite{B})
\begin{align}
& \sum_{l=1}^{r}\sum_{\mu_{l}=0}^{a_{l}-1}
\sum_{\substack{1=m_{l}-\sum_{1\leq k\not=l\leq r}\nu_{k}, \\ \nu_{1},\ldots,\nu_{r}\geq 0}}\prod_{1\leq u\not=l\leq r}\!\!\left\{(-1)^{\nu_{u}}\frac{(m_{u})_{\nu_{u}}}{\nu_{u}!}\varphi_{m_{u}+\nu_{u}}^{(I)}\left(a_{u}\frac{w_{l}+\mu_{l}}{a_{l}}-w_{u} \right)a_{u}^{m_{u}+\nu_{u}}\right\}  \nonumber \\
& \quad =\pi^{r-1}\sin\left(\frac{\pi{r}}{2}\right)\prod_{l=1}^{r}a_{l}\delta_{m_{l},1}.  
\end{align}
On the other hand, by putting $(j_{I},j_{I\hspace{-.1em}I})=(0,r)$, we obtain a cosecant analogue of Beck's result
\begin{align}
& \sum_{l=1}^{r}\sum_{\mu_{l}=0}^{a_{l}-1}
\sum_{\substack{1=m_{l}-\sum_{1\leq k\not=l\leq r}\nu_{k}, \\ \nu_{1},\ldots,\nu_{r}\geq 0}}\!\!\!\!\!\!\!\!\!\!\!\!\!\!\!\!(-1)^{\mu_{l}}
\prod_{1\leq u\not=l\leq r}\!\!\left\{(-1)^{\nu_{u}}\frac{(m_{u})_{\nu_{u}}}{\nu_{u}!}\varphi_{m_{u}+\nu_{u}}^{(I\hspace{-.1em}I)}\left(a_{u}\frac{w_{l}+\mu_{l}}{a_{l}}-w_{u} \right)a_{u}^{m_{u}+\nu_{u}}\right\}
=0.  
\end{align}
\end{exa}

\subsection{$\mathbf{w}=(0,\ldots,0)$ case}
In this subsection, we assume $\mathbf{w}=\mathbf{0}:=(0,\ldots,0)$ and $a_{1},\ldots,a_{r}$ are pairwise relatively prime. 
Under this condition, all poles of $a_{j}^{m_{j}}\varphi_{m_{j}}^{(J)}(a_{j}z)$ for each $j=1,\ldots,r$ on $\mathfrak{R}$ are 
$$
\frac{\mu_{j}}{a_{j}}\quad (j=1,\ldots,r,\,\, {\text{and}}\,\, \mu_{j}=0,\ldots,a_{j}-1).
$$
Further, $\delta_{\mathbb{Z}}(0)=1$ and for all $l=1,\ldots,r$, and $\mu_{j}=1,\ldots,a_{j}-1$,
$$
\delta_{\mathbb{Z}}\left(a_{l}\frac{\mu_{j}}{a_{j}}\right)=
\begin{cases}
    1 & (\text{if $l=j$}) \\
    0 & (\text{if $l\not =j$})
\end{cases}, \,\,\,\,\,\,
d_{v}^{(\mu_{v})}=
\begin{cases}
    r & (\text{if $\mu_{v}=0$}) \\
    1 & (\text{otherwise})
\end{cases}.
$$
Therefore, Theorems\,\ref{prop:main theorem}, \ref{prop:main theorem 2} degenerate to the following results. 
\begin{thm}
\label{prop:main theorem w zero}
Let 
\begin{align}
M_{n}^{(\mathbf{j})}(\mathbf{a},\mathbf{m})
:=&\sum_{\Lambda \in R_{0}}\sum_{(\nu_{k})_{k \in \Lambda^{c}}\in K_{n,\Lambda}^{-}}
\prod_{u \in \Lambda^{c}}(-1)^{m_{u}}\binom{m_{u}+\nu_{u}-1}{m_{u}-1}\alpha_{m_{u}+\nu_{u}}^{(\mathbf{j},u)}a_{u}^{m_{u}+\nu_{u}} \nonumber \\
=&\sum_{N=1}^{r-1}
\sum_{1\leq \lambda_{1} <\ldots<\lambda_{N} \leq r}
\sum_{\substack{n=|\mathbf{m}|-\sum_{k=1}^{N}(\nu_{\lambda_{k}}+m_{\lambda_{k}}), \\ \nu_{\lambda_{1}},\ldots,\nu_{\lambda_{N}}\geq 0}}
\prod_{u=1}^{N}(-1)^{m_{u}}\binom{m_{u}+\nu_{u}-1}{m_{u}-1}\alpha_{m_{u}+\nu_{u}}^{(\mathbf{j},u)}a_{u}^{m_{u}+\nu_{u}}. \nonumber
\end{align}
We obtain
\begin{align}
\Phi^{(\mathbf{j})}(z \,;\mathbf{a},\mathbf{m},\mathbf{0})&=\cos\left(\frac{\pi{r}}{2}\right)\pi^{r}\delta_{j_{I\hspace{-.1em}I},0}\prod_{l=1}^{r}a_{l}\delta_{m_{l},1}
+\sum_{n=1}^{|\mathbf{m}|}M_{n}^{(\mathbf{j})}(\mathbf{a},\mathbf{m})\varphi_{n}^{(J)}(z) \nonumber \\
& \quad +\sum_{n=1}^{\max_{j \in [r]}{\{m_{j}\}}}\sum_{l=1}^{r}\sum_{\mu_{l}=1}^{a_{l}-1}
\sum_{\substack{n=m_{l}-\sum_{1\leq k\not=l\leq r}\nu_{k}, \\ \nu_{1},\ldots,\nu_{r}\geq 0}}(-1)^{\mu_{l}\delta^{(\mathbf{j},l)}_{I\hspace{-.1em}I}} \nonumber \\
\label{eq:w zero}
& \quad \cdot \prod_{1\leq u\not=l\leq r}\left\{(-1)^{\nu_{u}}\frac{(m_{u})_{\nu_{u}}}{\nu_{u}!}\varphi_{m_{u}+\nu_{u}}^{(\mathbf{j},u)}\left(a_{u}\frac{\mu_{l}}{a_{l}} \right)a_{u}^{m_{u}+\nu_{u}}\right\}
\varphi_{n}^{(J)}\left(z-\frac{\mu_{l}}{a_{l}} \right).
\end{align}
\end{thm}
\begin{thm}
\label{prop:main theorem 2 w zero}
{\rm{(1)}}\,
We have
\begin{align}
& \sum_{l=1}^{r}\sum_{\mu_{l}=1}^{a_{l}-1}
\sum_{\substack{1=m_{l}-\sum_{1\leq k\not=l\leq r}\nu_{k}, \\ \nu_{1},\ldots,\nu_{r}\geq 0}}\!\!\!\!\!\!\!\!\!\!\!\!\!\!(-1)^{\mu_{l}\delta^{(\mathbf{j},l)}_{I\hspace{-.1em}I}}\!\!\!\prod_{1\leq u\not=l\leq r}\left\{(-1)^{\nu_{u}}\frac{(m_{u})_{\nu_{u}}}{\nu_{u}!}\varphi_{m_{u}+\nu_{u}}^{(\mathbf{j},u)}\left(a_{u}\frac{\mu_{l}}{a_{l}} \right)a_{u}^{m_{u}+\nu_{u}}\right\} \nonumber \\
\label{eq:w zero 1}
& \quad =\pi^{r-1}\sin\left(\frac{\pi{r}}{2}\right)\delta_{j_{I\hspace{-.1em}I},0}\prod_{l=1}^{r}a_{l}\delta_{m_{l},1}-M_{1}^{(\mathbf{j})}(\mathbf{a},\mathbf{m}).
\end{align}
In particular, for $\mathbf{m}=\mathbf{1}:=(1,\ldots,1)$, 
\begin{align}
\label{eq:general Zagier}
\sum_{l=1}^{r}\sum_{\mu_{l}=1}^{a_{l}-1}
(-1)^{\mu_{l}\delta^{(\mathbf{j},l)}_{I\hspace{-.1em}I}}\prod_{1\leq u\not=l\leq r}\left\{\varphi_{1}^{(\mathbf{j},u)}\left(a_{u}\frac{\mu_{l}}{a_{l}} \right)a_{u}\right\} =\pi^{r-1}\sin\left(\frac{\pi{r}}{2}\right)\delta_{j_{I\hspace{-.1em}I},0}\prod_{l=1}^{r}a_{l}-M_{1}^{(\mathbf{j})}(\mathbf{a},\mathbf{1}).
\end{align}
{\rm{(2)}}\,
For any $\mu \in \mathbb{Z}_{\geq 0}$ and $z_{0} \in \mathfrak{R}$, 
\begin{align}
& \sum_{n=1}^{\max_{j \in [r]}{\{m_{j}\}}}\sum_{l=1}^{r}\sum_{\mu_{l}=1}^{a_{l}-1}
\sum_{\substack{n=m_{l}-\sum_{1\leq k\not=l\leq r}\nu_{k} \\ \nu_{1},\ldots,\nu_{r}\geq 0}}\!(-1)^{\mu_{l}\delta^{(\mathbf{j},l)}_{I\hspace{-.1em}I}}\! \nonumber \\
& \quad \cdot \prod_{1\leq u\not=l\leq r}\left\{(-1)^{\nu_{u}}\frac{(m_{u})_{\nu_{u}}}{\nu_{u}!}\varphi_{m_{u}+\nu_{u}}^{(\mathbf{j},u)}\left(a_{u}\frac{\mu_{l}}{a_{l}} \right)a_{u}^{m_{u}+\nu_{u}}\right\}\!
A_{\mu}^{(J)}\left(z_{0} \,;1,n,\frac{\mu_{l}}{a_{l}}\right) \nonumber \\
& = -\cos\left(\frac{\pi{r}}{2}\right)\pi^{r}\delta_{\mu,0}\delta_{j_{I\hspace{-.1em}I},0}\prod_{l=1}^{r}a_{l}\delta_{m_{l},1} 
-\sum_{n=1}^{|\mathbf{m}|}M_{n}^{(\mathbf{j})}(\mathbf{a},\mathbf{m})A_{\mu}^{(J)}\left(z_{0} \,;1,n,0\right)  \nonumber \\
\label{eq:w zero 2.1}
& \quad  +\sum_{\Lambda \in R_{z_{0}}}\sum_{(\nu_{k})_{k \in \Lambda^{c}}\in K_{\mu, \Lambda}^{+}}
\prod_{l \in \Lambda}(-1)^{a_{l}z_{0}\delta^{(\mathbf{j},l)}_{I\hspace{-.1em}I}}
\prod_{u \in \Lambda^{c}}A_{\nu_{u}}^{(\mathbf{j},u)}(z_{0} \,;a_{u},m_{u},0).
\end{align}
In particular, by taking $z_{0}=0$, 
\begin{align}
& \sum_{n=1}^{\max_{j \in [r]}{\{m_{j}\}}}\sum_{l=1}^{r}\sum_{\mu_{l}=1}^{a_{l}-1}
\sum_{\substack{n=m_{l}-\sum_{1\leq k\not=l\leq r}\nu_{k} \\ \nu_{1},\ldots,\nu_{r}\geq 0}}\!(-1)^{\mu_{l}\delta^{(\mathbf{j},l)}_{I\hspace{-.1em}I}} \nonumber \\
& \quad \cdot \prod_{1\leq u\not=l\leq r}\left\{(-1)^{\nu_{u}}\frac{(m_{u})_{\nu_{u}}}{\nu_{u}!}\varphi_{m_{u}+\nu_{u}}^{(\mathbf{j},u)}\left(a_{u}\frac{\mu_{l}}{a_{l}} \right)a_{u}^{m_{u}+\nu_{u}}\right\}\!
(-1)^{\mu}\frac{(n)_{\mu}}{\mu !}\varphi_{n+\mu}^{(J)}\left(-\frac{\mu_{l}}{a_{l}}\right) \nonumber \\
& = -\cos\left(\frac{\pi{r}}{2}\right)\pi^{r}\delta_{\mu,0}\delta_{j_{I\hspace{-.1em}I},0}\prod_{l=1}^{r}a_{l}\delta_{m_{l},1} 
-\sum_{n=1}^{|\mathbf{m}|}M_{n}^{(\mathbf{j})}(\mathbf{a},\mathbf{m})
(-1)^{n}\binom{n+\mu -1}{n-1}\alpha_{\mu +n}^{(J)} \nonumber \\
\label{eq:w zero 2.2}
& \quad  +\sum_{\Lambda \in R_{0}}\sum_{(\nu_{k})_{k \in \Lambda^{c}}\in K_{\mu, \Lambda}^{+}}
\prod_{u \in \Lambda^{c}}\left\{(-1)^{m_{u}}\binom{m_{u}+\nu_{u}-1}{m_{u}-1}\alpha_{m_{u}+\nu_{u}}^{(\mathbf{j},u)}a_{u}^{m_{u}+\nu_{u}}\right\}.
\end{align}
\end{thm}
\begin{exa}
By putting $(j_{I},j_{I\hspace{-.1em}I})=(r,0)$ in (\ref{eq:general Zagier}), we obtain Zagier's result
\begin{align}
\pi^{r-1}\sum_{l=1}^{r}\sum_{\mu_{l}=1}^{a_{l}-1}
\prod_{u\not=l}\left\{\cot\left(\frac{{\pi}a_{u}\mu_{l}}{a_{l}}\right)a_{u}\right\}
& =\sin\left(\frac{{\pi}r}{2}\right)\pi^{r-1}\prod_{l=1}^{r}a_{l} \nonumber \\
\label{eq:Zagier}
& \quad 
-\sum_{N=1}^{r-1}
\sum_{1\leq \lambda_{1} <\ldots<\lambda_{N} \leq r}
\sum_{\substack{r-1-N=\sum_{k=1}^{N}\nu_{\lambda_{k}}, \\ \nu_{\lambda_{1}},\ldots,\nu_{\lambda_{N}}\geq 0}}
\prod_{u=1}^{N}\alpha_{\nu_{u}+1}^{(I)}a_{u}^{\nu_{u}+1}.
\end{align}
Further, if we consider the case of $(j_{I},j_{I\hspace{-.1em}I})=(0,r)$, then 
\begin{align}
\pi^{r-1}\sum_{l=1}^{r}\sum_{\mu_{l}=1}^{a_{l}-1}(-1)^{\mu_{l}}
\prod_{u\not=l}\left\{\csc\left(\frac{{\pi}a_{u}\mu_{l}}{a_{l}}\right)a_{u}\right\}
& =-\sum_{N=1}^{r-1}
\sum_{1\leq \lambda_{1} <\ldots<\lambda_{N} \leq r}
\sum_{\substack{r-1-N=\sum_{k=1}^{N}\nu_{\lambda_{k}}, \\ \nu_{\lambda_{1}},\ldots,\nu_{\lambda_{N}}\geq 0}} \nonumber \\
& \quad \cdot
\prod_{u=1}^{N}\alpha_{\nu_{u}+1}^{(I\hspace{-.1em}I)}a_{u}^{\nu_{u}+1}.
\end{align}
This is a cosecant version of Zagier's reciprocity law.
\end{exa}

\subsection{$r=2,\mathbf{m}=(1,1)$ case}
In this subsection, we assume $r=2,\mathbf{m}=(1,1)$ and $a_{1},a_{2}$ are relatively prime. 
For this simple case, we obtain more explicit expressions for our main results. 
\begin{thm}
Let $I\leq K_{1}\leq K_{2}\leq I\hspace{-.1em}I$, 
and $A_{1}$, $A_{2}$ denote integers for which $A_{1}a_{2}+A_{2}a_{1}=1$ holds. 
For the following case
\begin{equation}
\label{eq:sgn cond}
(K_{1},K_{2},J)=(I,I,I),\,\,(I,I\hspace{-.1em}I,I),\,\,(I,I\hspace{-.1em}I,I\hspace{-.1em}I),\,\,(I\hspace{-.1em}I,I\hspace{-.1em}I,I),\,\,(I\hspace{-.1em}I,I\hspace{-.1em}I,I\hspace{-.1em}I),
\end{equation}
we obtain
\begin{align}
& a_{1}a_{2}\varphi_{1}^{(K_{1})}(a_{1}z-w_{1})\varphi_{1}^{(K_{2})}(a_{2}z-w_{2}) \nonumber \\
& =-\pi^{2}a_{1}a_{2}\delta_{K_{1},I}\delta_{K_{2},I} \nonumber \\
& \quad +\delta_{\mathbb{Z}}(a_{1}w_{2}-a_{2}w_{1})\sgn_{2}^{(K_{1},K_{2})}((a_{1},a_{2}),(w_{1},w_{2}),(A_{1},A_{2}))\varphi_{2}^{(J)}(z-(A_{1}w_{2}+A_{2}w_{1})) \nonumber \\
& \quad +a_{2}\sum_{\mu_{1} =0}^{a_{1}-1}{}^{\prime}(-1)^{\mu_{1}\delta_{K_{1},I\hspace{-.1em}I}} \varphi_{1}^{(K_{2})}\left(a_{2}\frac{w_{1}+\mu_{1}}{a_{1}}-w_{2} \right)
\varphi_{1}^{(J)}\left(z-\frac{w_{1}+\mu_{1}}{a_{1}} \right) \nonumber \\
\label{eq:r2 case}
& \quad +a_{1}\sum_{\mu_{2} =0}^{a_{2}-1}{}^{\prime}(-1)^{\mu_{2}\delta_{K_{2},I\hspace{-.1em}I}}
\varphi_{1}^{(K_{1})}\left(a_{1}\frac{w_{2}+\mu_{2}}{a_{2}}-w_{1} \right)
\varphi_{1}^{(J)}\left(z-\frac{w_{2}+\mu_{2}}{a_{2}} \right). 
\end{align}
Here, the sums run over non-singular points and
\begin{align}
\sgn_{2}^{(K_{1},K_{2})}((a_{1},a_{2}),(w_{1},w_{2}),(A_{1},A_{2}))
&:=\sgn^{(K_{1})}(A_{1}w_{2}+A_{2}w_{1} \,;a_{1}+a_{2},w_{1}+w_{2})\delta_{K_{1},K_{2}} \nonumber \\
 & \quad +\sgn^{(K_{1})}(A_{1}w_{2}+A_{2}w_{1} \,;a_{1},w_{1}) \nonumber \\
 & \quad \cdot\sgn^{(K_{2})}(A_{1}w_{2}+A_{2}w_{1} \,;a_{2},w_{2})(1-\delta_{K_{1},K_{2}}). \nonumber
\end{align}
\end{thm}
\begin{proof}
The multiplicity free case 
$$
{\text{i.e.}}\,\,\,\frac{w_{1}+\mu_{1}}{a_{1}}\not=\frac{w_{2}+\mu_{2}}{a_{2}}\,\,\,(\mu_{1}=0,1,\ldots,a_{1}-1,\,\mu_{2}=0,1,\ldots,a_{2}-1)
$$
has been proved by some special cases of (\ref{eq:mutiplicity free}). 
Hence, it is enough to show another case. 
In the case, since $a_{1},a_{2}$ are relatively prime and $w_{1},w_{2} \in \mathfrak{R}$, there exists unique integers $\widetilde{\mu_{1}}\in \{0,1,\ldots,a_{1}-1\}$ and $\widetilde{\mu_{2}}\in \{0,1,\ldots,a_{2}-1\}$ such that 
$$
\rho_{0}:=\frac{w_{1}+\widetilde{\mu_{1}}}{a_{1}}=\frac{w_{2}+\widetilde{\mu_{2}}}{a_{2}},
$$
and 
\begin{align}
A_{1}w_{2}+A_{2}w_{1}&=\rho_{0}-(A_{1}\widetilde{\mu_{2}}+A_{2}\widetilde{\mu_{1}}), \nonumber \\ 
a_{1}(A_{1}w_{2}+A_{2}w_{1})-w_{1}
&=-a_{1}(A_{1}\widetilde{\mu_{2}}+A_{2}\widetilde{\mu_{1}})+\widetilde{\mu_{1}}, \nonumber \\ 
a_{2}(A_{1}w_{2}+A_{2}w_{1})-w_{2}
&=-a_{2}(A_{1}\widetilde{\mu_{2}}+A_{2}\widetilde{\mu_{1}})+\widetilde{\mu_{2}}, \nonumber \\
(a_{1}+a_{2})(A_{1}w_{2}+A_{2}w_{1})-(w_{1}+w_{2})
&=-(a_{1}+a_{2})(A_{1}\widetilde{\mu_{2}}+A_{2}\widetilde{\mu_{1}})+(\widetilde{\mu_{1}}+\widetilde{\mu_{2}}). \nonumber
\end{align}
Hence, from (\ref{eq:main result}), we have
\begin{align}
& a_{1}a_{2}\varphi_{1}^{(K_{1})}(a_{1}z-w_{1})\varphi_{1}^{(K_{2})}(a_{2}z-w_{2}) \nonumber \\
& =-\pi^{2}a_{1}a_{2}\delta_{K_{1},I}\delta_{K_{2},I} 
+(-1)^{\widetilde{\mu_{1}}\delta^{(\mathbf{j},1)}_{I\hspace{-.1em}I}+\widetilde{\mu_{2}}\delta^{(\mathbf{j},2)}_{I\hspace{-.1em}I}}\varphi_{2}^{(J)}(z-\rho_{0}) \nonumber \\
& \quad +a_{2}\sum_{\mu_{1} =0}^{a_{1}-1}{}^{\prime}(-1)^{\mu_{1}\delta_{K_{1},I\hspace{-.1em}I}}\varphi_{1}^{(K_{2})}\left(a_{2}\frac{w_{1}+\mu_{1}}{a_{1}}-w_{2} \right)
\varphi_{1}^{(J)}\left(z-\frac{w_{1}+\mu_{1}}{a_{1}} \right) \nonumber \\
& \quad +a_{1}\sum_{\mu_{2} =0}^{a_{2}-1}{}^{\prime}(-1)^{\mu_{2}\delta_{K_{2},I\hspace{-.1em}I}}\varphi_{1}^{(K_{1})}\left(a_{1}\frac{w_{2}+\mu_{2}}{a_{2}}-w_{1} \right)
\varphi_{1}^{(J)}\left(z-\frac{w_{2}+\mu_{2}}{a_{2}} \right). \nonumber
\end{align}
Here, we remark under the above five conditions (\ref{eq:sgn cond})
$$
(-1)^{\widetilde{\mu_{1}}\delta^{(\mathbf{j},1)}_{I\hspace{-.1em}I}+\widetilde{\mu_{2}}\delta^{(\mathbf{j},2)}_{I\hspace{-.1em}I}}
=(-1)^{\widetilde{\mu_{1}}\delta_{K_{1},I\hspace{-.1em}I}\delta_{K_{2},I\hspace{-.1em}I}+\widetilde{\mu_{2}}\delta_{K_{2},I\hspace{-.1em}I}}.
$$
Thus, by the definition of the signature (\ref{eq:def of sgn}) and the periodicity of $\varphi_{N}^{(J)}$ (\ref{eq:period}), 
\begin{align}
& \sgn_{2}^{(K_{1},K_{2})}((a_{1},a_{2}),(w_{1},w_{2}),(A_{1},A_{2}))\varphi_{2}^{(J)}(z-(A_{1}w_{2}+A_{2}w_{1})) \nonumber \\
& = \widetilde{\sgn}_{2}^{((K_{1},K_{2}),J)}((a_{1},a_{2}),(w_{1},w_{2}),(A_{1},A_{2}))
(-1)^{\widetilde{\mu_{1}}\delta_{K_{1},I\hspace{-.1em}I}\delta_{K_{2},I\hspace{-.1em}I}+\widetilde{\mu_{2}}\delta_{K_{2},I\hspace{-.1em}I}}\varphi_{2}^{(J)}(z-\rho_{0}), \nonumber
\end{align}
where 
\begin{align}
& \widetilde{\sgn}_{2}^{((K_{1},K_{2}),J)}((a_{1},a_{2}),(w_{1},w_{2}),(A_{1},A_{2})) \nonumber \\
&:=(-1)^{(A_{1}\widetilde{\mu_{2}}+A_{2}\widetilde{\mu_{1}})\{(a_{1}+a_{2})\delta_{K_{1},I\hspace{-.1em}I}+\delta_{J,I\hspace{-.1em}I}\}+\widetilde{\mu_{1}}\delta_{K_{1},I\hspace{-.1em}I}(1+\delta_{K_{2},I\hspace{-.1em}I})+\widetilde{\mu_{2}}(\delta_{K_{1},I\hspace{-.1em}I}+\delta_{K_{2},I\hspace{-.1em}I})}\delta_{K_{1},K_{2}} \nonumber \\
& \quad +(-1)^{(A_{1}\widetilde{\mu_{2}}+A_{2}\widetilde{\mu_{1}})(a_{1}\delta_{K_{1},I\hspace{-.1em}I}+a_{2}\delta_{K_{2},I\hspace{-.1em}I}+\delta_{J,I\hspace{-.1em}I})+\widetilde{\mu_{1}}\delta_{K_{1},I\hspace{-.1em}I}(1+\delta_{K_{2},I\hspace{-.1em}I})}(1-\delta_{K_{1},K_{2}}). \nonumber
\end{align}
Therefore, we claim that for the above five conditions (\ref{eq:sgn cond}), 
$$
\widetilde{\sgn}_{2}^{((K_{1},K_{2}),J)}((a_{1},a_{2}),(w_{1},w_{2}),(A_{1},A_{2}))=1
$$
and obtain the conclusion. 
\end{proof}
\begin{exa}
\noindent
{\rm{(0)}}\,{\rm{(Theorem\,2.4 in \cite{D})}}\,
$(K_{1},K_{2},J)=(I,I,I)$ case. 
\begin{align}
\cot{\pi (a_{1}z-w_{1})}\cot{\pi (a_{2}z-w_{2})}
&=-1-\frac{1}{a_{1}a_{2}}\delta_{\mathbb{Z}}(a_{1}w_{2}-a_{2}w_{1})\cot^{(1)}(\pi (z-(A_{1}w_{2}+A_{2}w_{1}))) \nonumber \\
& \quad 
+\frac{1}{a_{1}}{\sum_{\mu_{1}=0}^{a_{1}-1}}{}^{'}\cot\!\left(\!\pi \!\left(\!a_{2}\frac{w_{1}+\mu_{1}}{a_{1}}-w_{2}\!\right)\!\!\right)
\!\cot\!\left(\!\pi \!\left(\!z-\frac{w_{1}+\mu_{1}}{a_{1}}\!\right)\!\!\right) \nonumber \\
& \quad 
+\frac{1}{a_{2}}{\sum_{\mu_{2}=0}^{a_{2}-1}}{}^{'}\!\cot\!\left(\!\pi \!\left(\!a_{1}\frac{w_{2}+\mu_{2}}{a_{2}}-w_{1}\!\right)\!\!\right)
\!\cot\!\left(\!\pi \!\left(\!z-\frac{w_{2}+\mu_{2}}{a_{2}}\!\right)\!\!\right). \nonumber
\end{align}
{\rm{(1)}}\,
$(K_{1},K_{2},J)=(I,I\hspace{-.1em}I,I)$ case. 
\begin{align}
& \cot{\pi (a_{1}z-w_{1})}\csc{\pi (a_{2}z-w_{2})} \nonumber \\
& \quad 
=-\frac{(-1)^{a_{2}(A_{1}w_{2}+A_{2}w_{1})-w_{2}}}{a_{1}a_{2}}\delta_{\mathbb{Z}}(a_{1}w_{2}-a_{2}w_{1}) 
\cot^{(1)}(\pi (z-(A_{1}w_{2}+A_{2}w_{1}))) \nonumber \\
& \quad 
+\frac{1}{a_{1}}{\sum_{\mu_{1}=0}^{a_{1}-1}}{}^{'}\csc\!\left(\!\pi \!\left(\!a_{2}\frac{w_{1}+\mu_{1}}{a_{1}}-w_{2}\!\right)\!\!\right)
\!\cot\!\left(\!\pi \!\left(\!z-\frac{w_{1}+\mu_{1}}{a_{1}}\!\right)\!\!\right) \nonumber \\
\label{eq:generalized prot thm1.3}
& \quad 
+\frac{1}{a_{2}}{\sum_{\mu_{2}=0}^{a_{2}-1}}{}^{'}\!(-1)^{\mu_{2}}\cot\!\left(\!\pi \!\left(\!a_{1}\frac{w_{2}+\mu_{2}}{a_{2}}-w_{1}\!\right)\!\!\right)
\!\cot\!\left(\!\pi \!\left(\!z-\frac{w_{2}+\mu_{2}}{a_{2}}\!\right)\!\!\right). 
\end{align}
{\rm{(2)}}\,
$(K_{1},K_{2},J)=(I,I\hspace{-.1em}I,I\hspace{-.1em}I)$ case. 
\begin{align}
& \cot{\pi (a_{1}z-w_{1})}\csc{\pi (a_{2}z-w_{2})} \nonumber \\
& \quad 
=-\frac{(-1)^{a_{2}(A_{1}w_{2}+A_{2}w_{1})-w_{2}}}{a_{1}a_{2}}\delta_{\mathbb{Z}}(a_{1}w_{2}-a_{2}w_{1}) 
\csc^{(1)}(\pi (z-(A_{1}w_{2}+A_{2}w_{1}))) \nonumber \\
& \quad 
+\frac{1}{a_{1}}{\sum_{\mu_{1}=0}^{a_{1}-1}}{}^{'}\csc\!\left(\!\pi \!\left(\!a_{2}\frac{w_{1}+\mu_{1}}{a_{1}}-w_{2}\!\right)\!\!\right)
\!\csc\!\left(\!\pi \!\left(\!z-\frac{w_{1}+\mu_{1}}{a_{1}}\!\right)\!\!\right) \nonumber \\
\label{eq:generalized prot thm1.4}
& \quad 
+\frac{1}{a_{2}}{\sum_{\mu_{2}=0}^{a_{2}-1}}{}^{'}\!(-1)^{\mu_{2}}\cot\!\left(\!\pi \!\left(\!a_{1}\frac{w_{2}+\mu_{2}}{a_{2}}-w_{1}\!\right)\!\!\right)
\!\csc\!\left(\!\pi \!\left(\!z-\frac{w_{2}+\mu_{2}}{a_{2}}\!\right)\!\!\right). 
\end{align}
{\rm{(3)}}\,
$(K_{1},K_{2},J)=(I\hspace{-.1em}I,I\hspace{-.1em}I,I)$ case. 
\begin{align}
& \csc{\pi (a_{1}z-w_{1})}\csc{\pi (a_{2}z-w_{2})} \nonumber \\
& \quad 
=-\frac{(-1)^{(a_{1}+a_{2})(A_{1}w_{2}+A_{2}w_{1})-(w_{1}+w_{2})}}{a_{1}a_{2}}\delta_{\mathbb{Z}}(a_{1}w_{2}-a_{2}w_{1}) 
\cot^{(1)}(\pi (z-(A_{1}w_{2}+A_{2}w_{1}))) \nonumber \\
& \quad 
+\frac{1}{a_{1}}{\sum_{\mu_{1}=0}^{a_{1}-1}}{}^{'}\!(-1)^{\mu_{1}}\csc\!\left(\!\pi \!\left(\!a_{2}\frac{w_{1}+\mu_{1}}{a_{1}}-w_{2}\!\right)\!\!\right)
\!\cot\!\left(\!\pi \!\left(\!z-\frac{w_{1}+\mu_{1}}{a_{1}}\!\right)\!\!\right) \nonumber \\
\label{eq:generalized prot thm1.1}
& \quad 
+\frac{1}{a_{2}}{\sum_{\mu_{2}=0}^{a_{2}-1}}{}^{'}\!(-1)^{\mu_{2}}\csc\!\left(\!\pi \!\left(\!a_{1}\frac{w_{2}+\mu_{2}}{a_{2}}-w_{1}\!\right)\!\!\right)
\!\cot\!\left(\!\pi \!\left(\!z-\frac{w_{2}+\mu_{2}}{a_{2}}\!\right)\!\!\right).
\end{align}
{\rm{(4)}}\,
$(K_{1},K_{2},J)=(I\hspace{-.1em}I,I\hspace{-.1em}I,I\hspace{-.1em}I)$ case. 
\begin{align}
& \csc{\pi (a_{1}z-w_{1})}\csc{\pi (a_{2}z-w_{2})} \nonumber \\
& \quad 
=-\frac{(-1)^{(a_{1}+a_{2})(A_{1}w_{2}+A_{2}w_{1})-(w_{1}+w_{2})}}{a_{1}a_{2}}\delta_{\mathbb{Z}}(a_{1}w_{2}-a_{2}w_{1}) 
\csc^{(1)}(\pi (z-(A_{1}w_{2}+A_{2}w_{1}))) \nonumber \\
& \quad 
+\frac{1}{a_{1}}{\sum_{\mu_{1}=0}^{a_{1}-1}}{}^{'}\!(-1)^{\mu_{1}}\csc\!\left(\!\pi \!\left(\!a_{2}\frac{w_{1}+\mu_{1}}{a_{1}}-w_{2}\!\right)\!\!\right)
\!\csc\!\left(\!\pi \!\left(\!z-\frac{w_{1}+\mu_{1}}{a_{1}}\!\right)\!\!\right) \nonumber \\
\label{eq:generalized prot thm1.2}
& \quad 
+\frac{1}{a_{2}}{\sum_{\mu_{2}=0}^{a_{2}-1}}{}^{'}\!(-1)^{\mu_{2}}\csc\!\left(\!\pi \!\left(\!a_{1}\frac{w_{2}+\mu_{2}}{a_{2}}-w_{1}\!\right)\!\!\right)
\!\csc\!\left(\!\pi \!\left(\!z-\frac{w_{2}+\mu_{2}}{a_{2}}\!\right)\!\!\right). 
\end{align}
\end{exa}
(\ref{eq:generalized prot thm1.3}), (\ref{eq:generalized prot thm1.4}), (\ref{eq:generalized prot thm1.1}) and (\ref{eq:generalized prot thm1.2}) are generalizations 
of (\ref{eq:another prot thm1.3}), (\ref{eq:another prot thm1.4}), (\ref{eq:another prot thm1.1}) and (\ref{eq:another prot thm1.2}) respectively. 
Actually, by putting $w_{1}=w_{2}=0$, our results become Fukuhara's formulas. 
\begin{thm}
{\rm{(1)}}\,
\begin{align}
& a_{2}\sum_{\mu_{1} =0}^{a_{1}-1}{}^{\prime}(-1)^{\mu_{1}\delta_{K_{1},I\hspace{-.1em}I}}\varphi_{1}^{(K_{2})}\left(a_{2}\frac{w_{1}+\mu_{1}}{a_{1}}-w_{2} \right) \nonumber \\
\label{eq:r2 case 1}
& \quad +a_{1}\sum_{\mu_{2} =0}^{a_{2}-1}{}^{\prime}(-1)^{\mu_{2}\delta_{K_{2},I\hspace{-.1em}I}}\varphi_{1}^{(K_{1})}\left(a_{1}\frac{w_{2}+\mu_{2}}{a_{2}}-w_{1} \right)=0. 
\end{align}
{\rm{(2)}}\,
For any $\mu \in \mathbb{Z}_{\geq 0}$ and $z_{0} \in \mathfrak{R}$, 
\begin{align}
& a_{2}\sum_{\mu_{1} =0}^{a_{1}-1}{}^{\!\!\!\prime}(-1)^{\mu_{1}\delta_{K_{1},I\hspace{-.1em}I}}
\varphi_{1}^{(K_{2})}\left(a_{2}\frac{w_{1}+\mu_{1}}{a_{1}}-w_{2} \right)
A_{\mu}^{(J)}\left(z_{0} ;1,1,\frac{w_{1}+\mu_{1}}{a_{1}} \right) \nonumber \\
& \quad +a_{1}\sum_{\mu_{2} =0}^{a_{2}-1}{}^{\!\!\!\prime}(-1)^{\mu_{2}\delta_{K_{2},I\hspace{-.1em}I}}
\varphi_{1}^{(K_{1})}\left(a_{1}\frac{w_{2}+\mu_{2}}{a_{2}}-w_{1} \right)
A_{\mu}^{(J)}\left(z_{0} ;1,1,\frac{w_{2}+\mu_{2}}{a_{2}} \right) \nonumber \\
& \quad =\pi^{2}a_{1}a_{2}\delta_{K_{1},I}\delta_{K_{2},I}\delta_{\mu,0}  \nonumber \\
& \quad -\delta_{\mathbb{Z}}(a_{1}w_{2}-a_{2}w_{1})\sgn_{2}^{(K_{1},K_{2})}((a_{1},a_{2}),(w_{1},w_{2}),(A_{1},A_{2}))
A_{\mu}^{(J)}(z_{0} ;1,2,A_{1}w_{2}+A_{2}w_{1}) \nonumber \\
& \quad +\sgn^{(K_{1})}(z_{0} \,;a_{1},w_{1})A_{\mu +1}^{(K_{2})}(z_{0} ;a_{2},1,w_{2})
+\sgn^{(K_{2})}(z_{0} \,;a_{2},w_{2})A_{\mu +1}^{(K_{1})}(z_{0} ;a_{1},1,w_{1}) \nonumber \\
\label{eq:r2 case 2}
& \quad +\sum_{\nu =0}^{\mu}A_{\nu}^{(K_{1})}(z_{0} ;a_{1},1,w_{1})A_{\mu- \nu}^{(K_{2})}(z_{0} ;a_{2},1,w_{2}).
\end{align}
\end{thm}

\section{Concluding remarks}
We give Theorems\,\ref{prop:main theorem}, \ref{prop:main theorem 2}, which include as special cases reciprocity laws of various generalized Dedekind sums. 
Finally, as a future work, we raise a problem for an elliptic analogue of our main results.

Fukuhara and Yui derived the following formula in \cite{FY}. 
Fix a complex number $\tau$ with positive imaginary part. 
We put 
\begin{align}
\wp(z,\tau)&:=\frac{1}{z^{2}}+\sum_{\substack{\gamma \in \mathbb{Z}+\mathbb{Z}\tau  \\ \gamma \not=0}}\left\{\frac{1}{(z-\gamma)^{2}}-\frac{1}{\gamma^{2}}\right\}, \nonumber \\
\varphi(z,\tau)&:=\sqrt{\wp(z,\tau)-\wp\left(\frac{1}{2},\tau\right)}=\frac{1}{z}-\sum_{\nu \geq 0}\alpha_{\nu +1}(\tau)z^{\nu}. \nonumber
\end{align}
If $p$ and $q$ are relatively prime and $p+q$ is odd, then 
\begin{align}
\varphi(pz,\tau)\varphi(qz,\tau)&=-\frac{1}{pq}\varphi^{\prime}(z,\tau) \nonumber \\
& \quad +\frac{1}{p}\sum_{\substack{\mu, \lambda=0 \\ (\mu,\lambda) \not=(0,0)}}^{p-1}
\varphi\left(\frac{q(\mu +\lambda \tau)}{p},\tau\right)\varphi\left(z-\frac{\mu +\lambda \tau}{p},\tau\right) \nonumber \\
\label{eq:prot elliptic formula}
& \quad +\frac{1}{q}\sum_{\substack{\mu, \lambda=0 \\ (\mu,\lambda) \not=(0,0)}}^{q-1}
\varphi\left(\frac{p(\mu +\lambda \tau)}{q},\tau\right)\varphi\left(z-\frac{\mu +\lambda \tau}{q},\tau\right).
\end{align}
This formula can be regarded as an elliptic analogue of (\ref{eq:prot formula}).

Further, Egami \cite{E} provided the following reciprocity law which is an elliptic analogue of Zagier's reciprocity laws (\ref{eq:Zagier}). 
If $a_{1}, \ldots, a_{r} \in \mathbb{Z}_{\geq 0}$ are relatively prime and $a_{1}+\cdots+a_{r}$ is even, then 
\begin{equation}
\label{eq:egami}
\sum_{l=1}^{r}\sum_{\substack{\mu_{l},\lambda_{l} =0 \\ (\mu_{l},\lambda_{l}) \not=(0,0)}}^{a_{l}-1}
(-1)^{\lambda_{l}}\prod_{1\leq u\not=l\leq r}\left\{\varphi\left(a_{u}\frac{\mu_{l}+\lambda_{l}\tau}{a_{l}},\tau \right)a_{u}\right\} =-M(\tau; \mathbf{a}),
\end{equation}
where 
$$
M(\tau; \mathbf{a}):=\sum_{N=1}^{r-1}
\sum_{1\leq \lambda_{1} <\ldots<\lambda_{N} \leq r}
\sum_{\substack{r-1-N=\sum_{k=1}^{N}\nu_{\lambda_{k}}, \\ \nu_{\lambda_{1}},\ldots,\nu_{\lambda_{N}}\geq 0}}
(-1)^{N}\prod_{u=1}^{N}\{\alpha_{\nu_{u}+1}(\tau)a_{u}^{\nu_{u}+1}\}. \nonumber
$$

In this article, we obtain a generalization of (\ref{eq:prot formula}) and (\ref{eq:Zagier}). 
Therefore, we naturally consider the following problem. 
\begin{prob}
Give an elliptic analogue of Theorems\,\ref{prop:main theorem} and \ref{prop:main theorem 2}.   
\end{prob}

\section*{Acknowledgements}
We thank the anonymous referee and editor for their helpful advices. 

\bibliographystyle{amsplain}

\noindent Department of Pure and Applied Mathematics, 
Graduate School of Information Science and Technology, Osaka University, \\
1-5, Yamadaoka, Suita, Osaka, 565-0871, JAPAN.\\
E-mail: g-shibukawa@math.sci.osaka-u.ac.jp
\end{document}